\documentclass[11pt]{article}

\usepackage{url}
\usepackage[matrix,arrow,curve]{xy} 
\usepackage{a4wide}
\usepackage{amsmath,amssymb,amsthm}
\usepackage{tikz}
\usepackage{manfnt}
\usepackage{epsfig}
\usepackage{amssymb}
\usetikzlibrary{matrix}
\def\Z{{\Bbb Z}}
\def\N{{\Bbb N}}
\def\R{{\Bbb R}}
\def\diS1{\mbox{$\overrightarrow{\mathbb{S}^1}$}}
\def\diO1{\mbox{$\overrightarrow{\mathbb{O}^1}$}}
\def\diOo{\mbox{${\mathbb{O}^1}$}}
\def\diO2{\mbox{$\overrightarrow{\mathbb{O}^2}$}}

\newcommand\diSk[1]{\mbox{$\overrightarrow{\mathbb{S}^{#1}}$}}
\def\diSn{\diSk{n}}
\def\S1{\mbox{${\mathbb{S}^1}$}}

\newcommand\ab[0]{\textbf{Ab}}
\newcommand\sysh[2]{\overrightarrow{H}_{#1}(#2)}

\newcommand\trace[3]{\overrightarrow{\mathfrak{T}}(#1)(#2,#3)}
\newcommand\T[0]{\mathcal{T}}

\newcommand\M[0]{\mathcal{M}}

\usepackage{tkz-graph}
\GraphInit[vstyle = Shade]
\tikzset{
  LabelStyle/.style = { rectangle, rounded corners, draw,
                        minimum width = 2em, fill = yellow!50,
                        text = red, font = \bfseries },
  VertexStyle/.append style = { inner sep=5pt,
                                font = \Large\bfseries},
  EdgeStyle/.append style = {->, bend left} }

\newcommand\pth[1]{P#1}
\newcommand\diTC[1]{\overrightarrow{\sf TC}(#1)}

\newcommand\map[3]{#1 : #2 \rightarrow #3}
\newtheorem{lemma}{Lemma}
\newtheorem{theorem}{Theorem}
\newtheorem{proposition}{Proposition}
\newtheorem{corollary}{Corollary}
\newtheorem{example}{Example}
\newtheorem{definition}{Definition}

\newcommand\comment[1]{}

\theoremstyle{remark}
\newtheorem*{remark}{Remark}

\newcommand{\tc}{{\sf {TC}}}
\newcommand{\C}{\Bbb C}

\newcommand\ForAuthors[1]
 {\par\smallskip                     
  \begin{center}
   \fbox
   {\parbox{0.9\linewidth}
    {\raggedright--- #1}
   }
  \end{center}
  \par\smallskip                     %
 }        

\title{Directed topological complexity}
\author{
Eric Goubault\footnote{LIX, Ecole Polytechnique, CNRS, Universit\'e Paris-Saclay, 91128 Palaiseau, France, goubault@lix.polytechnique.fr}, Aur\'elien Sagnier\footnote{LIX \& CMAP, Ecole Polytechnique, CNRS, Universit\'e Paris-Saclay, 91128 Palaiseau, France, aurelien.sagnier@polytechnique.edu}, Michael Farber\footnote{School of Mathematical Sciences,
Queen Mary University of London, United Kingdom, 
m.farber@qmul.ac.uk}}

\begin{document}

\maketitle

\begin{abstract}
It has been observed that the very important motion planning problem of robotics mathematically speaking boils down to the  problem of finding a section to the path-space fibration, raising the notion of topological complexity, as introduced by M. Farber. 
The above notion fits the motion planning problem of robotics when there are no constraints on the actual control that can be applied to the physical apparatus. In many applications, however, a physical apparatus may have constrained controls, leading to constraints on its potential future dynamics. 
In this paper we adapt the notion of topological complexity to the case of directed topological spaces, which encompass such controlled systems, and also systems which appear in concurrency theory. We study its first properties, make calculations for some interesting classes of spaces, and show applications to a form of directed homotopy equivalence.
\end{abstract}


\section{Introduction}

In this paper we adapt the notion of topological complexity \cite{farber1}, \cite{farber}, to the case of directed topological spaces. Let us briefly motivate the interest in a notion of \lq\lq directed\rq\rq\, topological complexity. 
It has been observed that the very important
motion planning problem of robotics mathematically speaking boils down to the problem of
finding a section to the path-space fibration 
\begin{eqnarray}\label{1}\chi : \ X^I \rightarrow X \times X\end{eqnarray} where
$\chi(p)=(p(0),p(1))$; here $X^I$ denotes the space of all continuous paths
$p:I=[0,1]\to X$. If this section can be continuous, then the complexity $\tc(X)$ is the lowest
possible 
(equals to one), otherwise, $\tc(X)$ is defined as the minimal number of \lq\lq discontinuities\rq\rq\ that would encode
such a section. The notion of 
topological complexity is understandable both algorithmically, and topologically,
e.g. $\tc(X)=1$ is equivalent for $X$ to be contractible. 
Generally speaking, the topological complexity $\tc(X)$ is defined as the Schwartz
genus of the path space fibration.


The above definition fits the motion planning problem of robotics when there are 
no constraints on the actual control that can be applied to the physical apparatus
that is supposed to be moved from the state $a$ to the state $b$. In many applications, however, 
a physical apparatus may have dynamics that can be described as an ordinary
differential equation in the state variables $x \in \R^n$ in time $t$, and
parameterised by the control parameters $u \in \R^p$, 
\begin{eqnarray}\label{control}
\dot{x}(t)=f(t,x(t), u(t)).\end{eqnarray} 
The control parameters $u(t)$ are usually restricted 
to lie within a set $u\in U$. Equivalently, as is well-known,  one may describe the variety of trajectories of the control system 
(\ref{control}) is by using the language of differential inclusions,
\begin{eqnarray}\label{control1}
\dot{x}(t) \in F(t,x(t)),
\end{eqnarray}
where $F(t,x(t))$ is the set of all $f(t,x(t),u)$ with
$u \in U$. 
%
%
Under some well-investigated conditions this differential inclusion can be
proven to have solutions, at least locally.
Under these conditions, the set of solutions of the differential
inclusion (\ref{control1}) naturally forms a {\it directed space}, compare  \cite{grandisbook}, see also section  \ref{sec2} below. 
We observe in this paper that the motion planning problem of robotics in the presence of control constraints equates to finding
sections to the analogue of the path space fibration  (\ref{1}), i.e. the map taking
a d-path to the pair of its end points\footnote{That map would most likely
not qualify for being called a fibration in the directed setting.}. This material is developed in the following sections where we work in the generality of directed spaces. 
In particular we introduce
the notion of a directed homotopy equivalence which has precisely, and in a certain
non technical sense, minimally, the right properties with respect to the directed
version of topological complexity.

\section{Definitions}\label{sec2}


The context of a d-space was introduced in \cite{grandisbook};   
we will restrict ourselves later
to a more convenient category of d-spaces, that ought to be thought of as
some kind of cofibrant replacement of more general (but sometimes pathological) d-spaces.

\begin{definition}[\cite{grandisbook}]
A directed topological space, or a d-space $X=(X,PX)$ is a topological space equipped with
a set $PX$ of continuous maps $p:I \rightarrow X$ (where $I=[0,1]$ is the unit
segment with the usual topology inherited from $\R$), called directed paths or
d-paths, satisfying three axioms~: 
\begin{itemize}
\item every constant map $I\rightarrow X$ is directed;
\item $PX$ is closed under composition with continuous non-decreasing maps from $I\to I$;
\item $PX$ is closed under concatenation.
\end{itemize}
\end{definition}

Note that for a d-space $X$, the paths space $PX$ is a topological space, equipped with the compact-open topology. 

A map $f: X\to Y$ between d-spaces is {\it a d-map} if it is continuous and for any directed path $p\in PX$ 
the path $f\circ p:I\to Y$ belongs to $PY$. In other words we require that $f$ preserves directed paths. 


\begin{remark}
\label{saturation}
Given a topological space $X$ equipped with a set $D$ of paths $p : I \rightarrow X$, closed under concatenation and such that the union of the images
$p(I)$, for $p \in D$ is $X$, we call saturation $\overline{D}$
of $D$
the smallest set of paths containing $D$ that forms a d-structure on $X$. The saturation of $D$ is just made of all composites of path of $D$ with continuous and
non-decreasing maps from $I$ to $I$. 
\end{remark}

\paragraph{d-spaces in control theory.} Consider a differential inclusion 
\begin{equation}
\dot{x} \in F(x) 
\label{diffincl}
\end{equation}
\noindent where $F$ is a map from $\R^n$ to $\wp(\R^n)$, 
the set of all subsets of $\R^n$. 
A function $x: [0, \infty)\rightarrow \mathbb{R}^n$ 
 is a {\em solution} of inclusion (\ref{diffincl})
if $x$ is absolutely continuous  and for almost all $t\in\mathbb{R}$
one has $\dot{x}(t)\in F(x(t)),$ see \cite{Aubin}. 
In general, there can be many solutions to a differential inclusion. 

\begin{lemma} \cite{Aubin} 
Suppose a set-valued map $F: \mathbb{R}^n\leadsto\mathbb{R}^n$ is 
an upper semicontinuous function of $x$ and such that the set $F(x)$ is closed and convex for all $x$. 
Then there exists a solution to Equation
(\ref{diffincl}) defined on an open interval of time. 
\end{lemma}

Consider a smooth manifold $X$ and an upper semicontinuous set-valued mapping $x\mapsto F(x)$ where for $x\in X$ the image $F(x)$ is a convex cone contained in the tangent space to $X$ at point $x$, i.e. $F(x)\subset T_xX$. 
Let $PX$ denote the saturation of the set of all solutions to the differential inclusion $\dot x \in F(x)$.  Then the pair $(X, PX)$ is a d-space.

%
%

\begin{figure}[h]
\begin{center}
\begin{tikzpicture}[auto,scale = 1]
    \draw (0,0) rectangle (2.5,2.5);
    \draw [fill = gray!50,draw = gray!50] (0.5,0.5) rectangle (1,1);
    \draw [fill = gray!50,draw = gray!50] (0.5,1.5) rectangle (1,2);
    \draw [fill = gray!50,draw = gray!50] (1.5,0.5) rectangle (2,1);
    \draw [fill = gray!50,draw = gray!50] (1.5,1.5) rectangle (2,2);
    \draw [dotted] (1,0) to (1,2.5);
    \draw [dotted] (0.5,0) to (0.5,2.5);
    \draw [dotted] (1.5,0) to (1.5,2.5);
    \draw [dotted] (2,0) to (2,2.5);
    \draw [dotted] (0,1) to (2.5,1);
    \draw [dotted] (0,0.5) to (2.5,0.5);
    \draw [dotted] (0,1.5) to (2.5,1.5);
    \draw [dotted] (0,2) to (2.5,2);
    \node (0) at (-0.2,-0.2) {\scriptsize{\textbf{0}}};
    \node (1) at (2.7,2.7) {\scriptsize{\textbf{1}}};
    \draw (0.5,-0.1) to (0.5,0.1);
    \node (P1a) at (0.5,-0.3) {\scriptsize{$Pa$}};
    \draw (1,-0.1) to (1,0.1);
    \node (V1a) at (1,-0.3) {\scriptsize{$Va$}};
    \draw (1.5,-0.1) to (1.5,0.1);
    \node (P1ap) at (1.5,-0.3) {\scriptsize{$Pa$}};
    \draw (2,-0.1) to (2,0.1);
    \node (V1ap) at (2,-0.3) {\scriptsize{$Va$}};
    \draw (-0.1,0.5) to (0.1,0.5);
    \node (P2a) at (-0.4,0.5) {\scriptsize{$Pa$}};
    \draw (-0.1,1) to (0.1,1);
    \node (V2a) at (-0.4,1) {\scriptsize{$Va$}};
    \draw (-0.1,1.5) to (0.1,1.5);
    \node (P2ap) at (-0.4,1.5) {\scriptsize{$Pa$}};
    \draw (-0.1,2) to (0.1,2);
    \node (V2ap) at (-0.4,2) {\scriptsize{$Va$}};
    \draw [thick] (0,0) .. controls (0.2,1.25)  .. (1.25,1.25);
    \draw [thick] (0,0) .. controls (1.25,0.2)  .. (1.25,1.25);
    \draw [thick] (2.5,2.5) .. controls (1.25,2.3)  .. (1.25,1.25);
    \draw [thick] (2.5,2.5) .. controls (2.3,1.25)  .. (1.25,1.25);
    \draw [thick] (0,0) .. controls (0.1,2.5)  .. (2.5,2.5);
    \draw [thick] (0,0) .. controls (2.5,0.1)  .. (2.5,2.5);
    \node (b) at (0,-1) {};
  \end{tikzpicture}
\end{center}
\caption{The semantics of $Pa.Va.Pb.Vb | Pa.Va.Pb.Vb$.}
\label{sem1}
\end{figure}
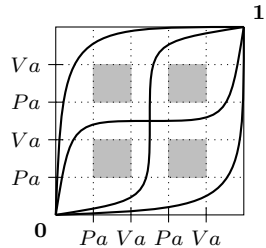
\paragraph{d-spaces in concurrency and distributed systems theory.}
The semantics of concurrent and distributed systems can be given in terms
of d-spaces, more specifically in terms of geometric realizations \cite{bookdag} of certain
pre-cubical sets. As an example, consider the following
concurrent program, made of two processes $T_1$, $T_2$, 
and two binary semaphores $a$, $b$, i.e. 
resources, that can only be accessed locked by one of
the two processes \cite{bookdag} at a time~: $T_1=Pa.Va.Pb.Vb$, $T_2=Pa.Va.Pb.Vb$, in 
the notations. This means that process $T_1$ is locking
$a$ ($Pa$), then relinquishing the lock on $a$ ($Va$), then locking 
$b$ ($Pb$), and finally relinquishing the lock of $b$ ($Vb$). Process
$T_2$ does the same sequence of actions. The semantics of this concurrent
program is depicted in Figure \ref{sem1}~: it is a partially ordered
space $X$, i.e. a topological space with a global order $\leq$, closed in 
$X \times X$. Its d-space structure is given by choosing dipaths to be
paths $p \ : \ I \rightarrow X$ such that $p$ is non-decreasing. 
A number of such dipaths are depicted in Figure \ref{sem1}.


\paragraph{The d-paths map.}
In what follows, we will be particularly concerned with the following map~: 
\begin{definition}
Let $(X, PX)$ be a d-space. Define the {\em d-paths map}
$$\chi : PX \rightarrow X \times X$$ by  
$\chi(p)=(p(0),p(1))$ where $p \in PX$.
\end{definition}

This map is analogous to the classical path-space fibration (\ref{1}); 
the essential distinction is that in the directed setting $\chi$, as defined above, is not necessary a fibration.

Since $PX$  contains only directed paths, the image of $\chi$ is a subset of $X \times X$, 
denoted $$\Gamma_X=\{ (x,y)\in X\times X \ | \ \exists p \in PX, \ p(0)=x, \ p(1)=y \ \}.$$
In the classical case, one do not need to force the restriction to the image
of the path space fibration since the notions of contractibility and path-connectedness
are simple enough to be defined separately. In the directed setting, d-contractibility,
and \lq\lq d-connectedness\rq\rq are not simple notions and will be defined 
here through the study of the d-path space map. 


\paragraph{Notations:}
For $a, b\in X$, the symbol $PX(a,b)$ 
will denote the subspace of $PX$ 
consisting of all d-paths from the point $a \in X$
to the point $b \in X$. We denote by *  the concatenation map 
$$PX(a,b)\times PX(b,c)\to PX(a,c).$$ 
Note that $PX(a,b)$ is non-empty
if and only if $(a,b) \in \Gamma_X$. 

Any d-map $f :  X \to Y$ induces continuous maps $\Gamma f: \Gamma_X \rightarrow \Gamma_Y$
and 
$Pf \ : \ PX \rightarrow PY$, such that the diagram 
$$
\begin{array}{ccc}
PX &\stackrel{Pf}\to & PY\\
\downarrow \chi_X && \downarrow\chi_Y\\
\Gamma_X& \stackrel{\Gamma f}\to & \Gamma_Y
\end{array}
$$
commutes. 
%
%
%

\section{Directed topological complexity}

Let $(X, PX)$ be a d-space such that $X$ is an Euclidean Neighbourhood Retract (ENR). 
\begin{definition}\label{def1}
The directed topological complexity $\diTC{X, PX}$ 
of a d-space $(X,PX)$ 
is the minimum number $n$ (or $\infty$
if no such $n$ exists) such
that there exists a map $s : \Gamma_X \rightarrow PX$ (not necessarily continuous) and 
$\Gamma_X$ can be partitioned into $n$ ENRs
%
$$\Gamma_X = F_1\cup F_2\cup \dots\cup F_n, \quad F_i\cap F_j =\emptyset,\quad i\not= j,$$
such that 
\begin{itemize}
\item $\chi \circ s = Id$, i.e. $s$ is a (non-necessarily continuous) section of $\chi$;
\item $s|_{F_i} : F_i \rightarrow PX$ is continuous. 
\end{itemize}
A collection of such ENRs, $F_1,\ldots,F_n$, with $n$ equal to the directed topological complexity
of $X$ is called a patchwork. 
\end{definition}

\paragraph{Example in control theory.}
As in \cite{farber}, a motion planner, for the dynamics described by the 
differential inclusion (\ref{diffincl}) is a section of the d-paths map  produced by
the differential inclusion. A section $s: \Gamma_X \to PX$ associates to any pair of points $(x, y)\in \Gamma_X$ an \lq\lq admissible\rq\rq\
path $s(x, y)=\gamma\in PX$ with $\gamma(0)=x$ and $\gamma(1)=y$. 

\paragraph{Example in concurrency and distributed systems theory.}
Examine again Figure \ref{sem1}; a section of $\chi$ is just a scheduler
for the actions of the processes $T_1$ and $T_2$. 

\vskip .2cm

In the theory of usual (i.e. undirected)  topological complexity \cite{farber1}, \cite{farber}, there are several other equivalent definitions, for example 
the topological complexity $\tc(X)$ is also the minimal cardinality of the covering of $X\times X$
by open (resp. closed) sets admitting continuous sections; moreover, the book \cite{farber} contains four different definitions of $\tc(X)$ 
leading to the equivalent notions of $\tc(X)$. 
In the directed case, however, the definitions with open or closed covers lead to notions which can be distinct between themselves as well as distinct from the notion with the ENR partitions given above. 


\begin{example}\label{ex1}{\rm 
Consider the interval $I=[0, 1]$ with the d-structure given by the set of all non-decreasing paths, i.e. $p: [0, 1]\to [0,1]$ such that 
$p(t)\le p(t')$ for any $t\le t'$.  The space $\Gamma_I$ is $\{(x, y); x\le y\}$ and the map $\chi: PI\to \Gamma_I$ admits a continuous section $$s(x, y)(t)=(1-t)x+ty$$ where $t\in [0,1].$ Hence $\diTC I=1$.

Note that in this example the space $\Gamma_I$ is contractible and the map 
$\chi$ is a fibration with a contractible fibre. 
}
\end{example}

\begin{example}\label{ex2}
\label{directedcircle}{\rm 
Let us consider the directed circle $\diS1$
shown on the figure below:
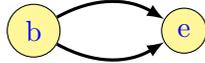
\begin{center}
\begin{tikzpicture}
\tikzstyle{cercle}=[circle,draw,fill=yellow!50,text=blue]
 \tikzstyle{fleche}=[->,>=latex,very thick]
 \node[cercle] (b) at (0,0) {b};
 \node[cercle] (e) at (2,0) {e};
  \draw[fleche] (b)to[bend left](e);
  \draw[fleche] (b)to[bend right](e);
\end{tikzpicture}
\end{center}
It is a directed graph homeomorphic to the circle $S^1$ which is the union of two directed intervals $I_+\cup I_-$; the d-paths of 
$\diS1$ 
are the d-paths lying in one of the intervals $I_\pm$. We see that 
$P(\diS1)
= P(I_+)\cup P(I_-)$ and 
$P(I_+)\cap P(I_-)$ is a 2-point set containing the two constant paths $p_b(t)\equiv b$ and $p_e(t)\equiv e$. Similarly, one has $\Gamma_{\diS1}=\Gamma_{I_+}\cup \Gamma_{I_-}$ and the intersection 
$\Gamma_{I_+}\cap \Gamma_{I_-}$ is a 3 point set $\{(b,b), (b, e), (e, e)\}$. Since each of the sets $\Gamma_{I_{\pm}}$ is contractible we obtain that $\Gamma_{\diS1}$ is homotopy equivalent to the wedge $S^1\vee S^1$. 

Next we observe that the map $\chi: P\diS1 \to \Gamma_{\diS1}$ {\it admits no continuous section over any neighbourhood $U$ of the point 
$(b, e)\in \Gamma_{\diS1}$}. To show this one notes that the preimage $\chi^{-1}(b, e)$ has two connected components, one of which consists of the d-paths lying in $I_+$ and the other of the d-paths lying in $I_-$. Any open set $U\subset \Gamma_{\diS1}$ containing $(b, e)$ 
must contain a pair $(x^+, y^+)\in \Gamma_{I_+}$ and a pair $(x^-, y^-)\in \Gamma_{I_-}$, arbitrarily close to $(b, e)$. Moreover, we may 
find two sequences $(x^\pm_n, y^\pm_n)\in \Gamma_{I_\pm}$ of points converging to $(b, a)$ and the limits of any section over $U$ 
along these sequences would land in different connected component of $\chi^{-1}(b, e)$. 
Hence, we obtain $\diTC {\diS1} \ge 2$. On the other hand, we may represent $\Gamma_{\diS1}$ as the
union $$\Gamma_{\diS1}=F_1\cup F_2$$ where $F_1=\Gamma_{I_+}$ and $F_2=\Gamma_{I_-}-\{(b, e), (b,b), (e,e)\}$ and using the previous example we see that over each of the sets $F_1, F_2$ there exists a continuous section of $\chi$. Hence we obtain 
\begin{eqnarray}
\diTC {\diS1} =2. 
\end{eqnarray}
}
\end{example}

\section{Regular d-spaces}

\begin{definition}\label{def4}
A d-space $(X, PX)$ will be called regular if one can find a partition $$\Gamma_X=F_1\cup F_2\cup \dots\cup F_n, \quad n=\diTC X$$ 
 onto ENRs such that the map $\chi$ admits a continuous section over each $F_i$ and, additionally, the sets $\mathop{\cup}\limits_{i=1}^r F_i$ are closed for any $r=1, \dots, n$. 
\end{definition}

Note the following property of the sets which appear in Definition \ref{def4}:
\begin{eqnarray}\label{property}
\overline F_i \cap F_{i'} = \emptyset \quad \mbox{for}\quad i<i'.
\end{eqnarray}

In the \lq\lq undirected\rq\rq\, theory of  $\tc(X)$ this property is automatically satisfied, see Proposition 4.12 of \cite{farber}. 

All examples of d-spaces which appear in this paper are regular. At present we know of no examples of d-spaces which are not regular; we plan to address this question in more detail elsewhere. 

\begin{example} \label{example3} {\rm The directed circle $\diS1$ is regular as follows from the construction of Example \ref{directedcircle}.
}
\end{example}

The Cartesian product of d-spaces $(X, PX)$ and $(Y, PY)$ has a natural d-space structure. Any path 
$\gamma: [0,1]\to X\times Y$ has the form $\gamma(t)=(\gamma_X(t), \gamma_Y(t))$ and we declare $\gamma$ to be directed if 
its both coordinates are directed, i.e. $\gamma_X\in PX$ and $\gamma_Y\in PY$. Note that $\Gamma_{X\times Y}=\Gamma_X\times\Gamma_Y$.

\begin{proposition}\label{prop1}
If the d-spaces $(X_i, PX_i)$ are regular, where $i=1, 2,\dots, k$, then 
\begin{eqnarray}
\diTC {X_1\times X_2\times \dots\times X_k} -1 \le \sum_{i=1}^k \left[\diTC {X_i} -1\right].
\end{eqnarray}
\end{proposition}
\begin{proof}
Denote $\diTC {X_i} = n_i+1$ and let 
$$\Gamma_{X_i} = F_0^i\cup F_1^i\cup \dots\cup F_{n_i}^i$$
be a partition as in the Definition \ref{def4}, i.e. each set $F_j^i$ is an ENR, the map $\chi$ admits a continuous section over $F_j^i$ 
and each union $F^i_0\cup \dots \cup F^i_r$ is closed, $r=0, \dots, n_i$. Denoting $X=\prod\limits_{i=1}^k X_i$ and 
identifying the space $\Gamma_{X}$ with the product $\prod_{i=1}^k \Gamma_{X_i}$, we see that the sets 
$$F_{j_1}^1 \times F_{j_2}^2 \times \dots\times F_{j_k}^k$$
form a ENR partition of $\Gamma_{X}$, where each index $j_s$ runs through $0, 1, \dots, n_s$. The continuous sections 
$F_{j_s}^s\to PX_s$,
where $s=1, \dots, k$, 
obviously produce  continuous sections
$$\sigma_{j_1 j_2\dots j_s}: F_{j_1}^1 \times F_{j_2}^2 \times \dots\times F_{j_k}^k\to PX.$$
Consider the sets
\begin{eqnarray}\label{union}
\bigcup_{j_1+\dots+j_k=j} F_{j_1}^1 \times F_{j_2}^2 \times \dots\times F_{j_k}^k =G_j\subset \Gamma_X,
\end{eqnarray}
with $j=0, 1, \dots, N$, where $N=n_1+n_2+\dots+n_k.$ 
We observe that the terms of the union (\ref{union}) are pairwise disjoint and open in $G_j$ (due to (\ref{property})) and hence the collection
of continuous maps  $\sigma_{j_1 j_2\dots j_s}$ defines a continuous section $G_j\to PX$. This proves that $\diTC X \le N+1$ as claimed. 

\end{proof}

\begin{corollary} The directed torus $(\diS1)^n$ satisfies $\diTC {(\diS1)^n} \le n+1$. 
\end{corollary}
\begin{proof}
This follows from Proposition \ref{prop1} and \ref{example3}. 
\end{proof}

%

\begin{definition}
We say that a d-space $X$ is strongly connected if $\Gamma_X =X\times X$. 
\end{definition}

In other words, in a strongly connected d-space $X$ 
for any pair $(x,y)$ in $X \times X$
there exists a directed path $\gamma\in PX$ with $\gamma(0)=x$, $\gamma(1)=y$.

\begin{proposition}
\label{SCD}
For any strongly connected d-space $X$ one has $\tc(X)\leq\diTC{X}$.
\end{proposition}
 
\begin{proof} Let $X$ be strongly connected and let $\Gamma_X=X\times X=F_1\cup F_2\cup \dots\cup F_n$ be a partition into the 
ENRs as in Definition \ref{def1} with $n=\diTC X$. Then the same partition can serve for the path space fibration $X^I\to X\times X$ which implies our result. 
%
\end{proof}

\begin{example}\label{ex4}
\label{diS11}{\rm 
Consider {\it the directed loop} $\diOo$ which can be defined as the unit circle $$S^1=\{z\in \C; |z|=1\}\subset \C$$ with the d-structure described below. Any continuous path 
$\gamma: [0,1]\to S^1$ can be presented in the form $\gamma(t)=\exp(i\phi(t))$ where the function $\phi: [0,1]\to \R$ is defined uniquely up to adding an integer multiple of $\pm 2\pi$. We declare a path $\gamma$ to be positive if the function $\phi(t)$ is nondecreasing. 
It is obvious that the obtained d-space is strongly connected. Hence, using Proposition \ref{SCD}, we obtain $\diTC{\diOo} \ge \tc(S^1)=2$, . On the other hand, we can partition $S^1\times S^1=F_1\cup F_2$ where 
$F_1=\{(z_1, z_2)\in S^1\times S^1; z_1=z_2\}$ and 
$F_2=\{(z_1, z_2)\in S^1\times S^1; z_1\not=z_2\}$. It is clear that we obtain a section of $\chi$ over $F_1$ by assigning the constant path at $z$ for any pair $(z, z)\in F_1$. 
A continuous section of $\chi$ over $F_2$ can be defined as follows by moving $z_1$ along the circle in the positive direction towards $z_2$ with constant velocity. 
%
%
We conclude that 
\begin{eqnarray}\label{eq2}
\diTC {\diOo} =2.
\end{eqnarray}
Besides, we see that the directed loop $\diOo$ is regular. 

\begin{corollary} 
One has, $$\diTC {(\diOo)^n}= n+1,$$ i.e. the directed topological complexity of the directed $n$-dimensional torus $(\diOo)^n$ equals $n+1$. 
\end{corollary}

\begin{proof} 
First we apply (\ref{eq2}) and Proposition \ref{prop1} to obtain the inequality $\diTC {(\diOo)^n}\le n+1$. Next we observe that $(\diOo)^n$ is strongly connected and, by Proposition \ref{SCD}, $\diTC {(\diOo)^n} \ge \tc((S^1)^n)=n+1$. 
\end{proof}
%
%
%
%
}
\end{example}
\section{Directed graphs}

Let $G$ be a directed connected graph, i.e. each edge of $G$ has a specified orientation. One naturally defines a d-structure on $G$ 
as follows. Each edge of $G$ can be identified either with the directed interval $I$ (see Example \ref{ex1}) or with the loop $\diOo$ 
(see Example \ref{ex4}) and \lq\lq small directed paths\rq\rq, i.e. the paths lying on an edge, are the directed paths specified in Example \ref{ex1} and Example \ref{ex4}. In general, the directed paths of $G$ are concatenations of small directed paths. 

For a directed graph $G$ the set $\Gamma_G$ has the following property: if a pair $(x, y)$ belongs to $\Gamma_G$ where $x $ is an internal point of an edge $e$ and $y\notin e$ then all pairs $(x', y)$ also belong to $\Gamma_G$ where $x'\in {\rm Int} (e)$.  

\begin{proposition}\label{generalgraphs}
$\diTC{G}\leq 3$.
\end{proposition}

\begin{proof}
Consider the following partition $\Gamma_G=F_1\cup F_2\cup F_3$ where 
\begin{itemize}
\item $F_1$ is the set of pairs of vertices $(\alpha_i,\alpha_j)$ of $G$ which are in $\Gamma_G$;
\item $F_2$ is the set of pairs $(x,y) \in \Gamma_G$ made of a vertex, and the interior of
an arc;
\item $F_3$ is the set of pairs  $(x,y)\in \Gamma_G$ with $x$ and $y$ lying in the interiors of arcs. 
\end{itemize}
For each pair of vertices $(\alpha_i, \alpha_j)\in \Gamma_G$ fix a directed path $\gamma_{ij}$ from $\alpha_i$ to $\alpha_j$. 
This defines a section of $\chi$ over $F_1$. Note that all pairs $(\alpha_i, \alpha_i)$ belong to $\Gamma_G$ and the path 
$\gamma_{ii}$ can be chosen to be constant. 

%
%

Consider now an oriented edge $e$ and a vertex $\alpha_j$ such that $(x, \alpha_j)\in \Gamma_G$ for an internal point 
$x\in {\rm Int}(e)$. Let $\alpha_i$ be the end point of $e$ and let $\gamma_{x, \alpha_i}$ denote the constant velocity path along $e$ from $x$ to $\alpha_i$. 
 A continuous section of $\chi$ over ${\rm Int}(e)\times \alpha_j$ can be defined as $(x, \alpha_j)\mapsto \gamma_{x, \alpha_i}\star \gamma_{ij}$ where $\ast$ stands for concatenation. A continuous section over $\alpha_j \times {\rm Int}(e)$ can be defined similarly, and hence we have a continuous section of $\chi$ over $F_2$. 
 
Finally we describe a continuous section of $\chi$ over $F_3$. Consider two oriented edges $e$ and $e'$ where we shall first assume that 
$e\not= e'$. Let $\alpha$ denote the end point of $e$ and $\beta$ denote the initial point of $e'$. We define a section of $\chi$ by 
$$(x, y)\mapsto \gamma_{x, \alpha}\ast \gamma_{\alpha \beta}\ast \gamma_{\beta y}$$ for $x\in {\rm Int}(e)$ and $y\in {\rm Int}(e')$. 
 Here $\gamma_{x\alpha}$ denotes a constant velocity directed path along $e$ connecting $x$ to $\alpha$; the path 
 $\gamma_{\beta y}$ is defined similarly and $\gamma_{\alpha\beta}$ is a positive path from $\alpha$ to $\beta$. 
 
 Finally we consider the case when $e=e'$. For a pair $(x, y)\in \Gamma_G$ with $x, y\in {\rm Int}(e)$ we define the section by 
 $(x, y)\mapsto \gamma_{xy}$ where $\gamma_{xy}$ is a constant velocity path along $e$ from $x$ to $y$. 
 
 All the partial sections described above over various parts of $F_3$ obviously combine into a continuous section over $F_3$. 
 
%
%
%
%
\end{proof}

The following example shows that the directed topological complexity can be smaller than the usual complexity. 

\begin{example}{\rm 
Consider the following graph~: 
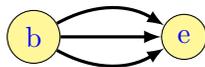
\begin{center}
\begin{tikzpicture}
\tikzstyle{cercle}=[circle,draw,fill=yellow!50,text=blue]
 \tikzstyle{fleche}=[->,>=latex,very thick]
 \node[cercle] (a) at (0,0) {b};
 \node[cercle] (b) at (2,0) {e};
  \draw[fleche] (a)to[bend left](b);
  \draw[fleche] (a)to[bend right](b);
  \draw[fleche] (a)to(b);
\end{tikzpicture}
\end{center}
A patchwork for $\Gamma_G$~: $F_1=\{(b,e)\}$ and $F_2=\Gamma_G \backslash
F_1$. 
We thus have $\diTC{G}=2$ (here again, it is easy to see that there is no global section). But $\tc(G)=3$.
}
\end{example}

However in the special case of strongly connected graphs, the directed and classical
topological complexity coincide: 

\begin{proposition}
Let $G$ be a strongly connected directed graph. Then 
$$\diTC{G}=\tc(G)=\min(b_{1}(G),2)+1.$$
\end{proposition}
\begin{proof}
By \cite{farber}, we know that $TC(G)=\min(b_1(G),2)+1$. As $G$ is strongly
connected, we have $\diTC{G}\geq \tc(G)=\min(b_1(G),2)+1$, see Proposition \ref{SCD}.
To prove that we have in fact an equality consider the following cases: 
\begin{itemize}
\item $b_{1}(G)=0$. Since $G$ is contractible and strongly connected, $G$ must be a single point. 
Then $\diTC G =1$ and the result follows. 
\item $b_{1}(G)=1$. It is easy to see that in this case $G$ must be a cycle, i.e. $G$ has $n$ vertices $v_1, v_2, \dots, v_n$ and $n$ oriented edges 
$e_1, e_2, \dots, e_n$ where $e_i$ connects $v_i$ with $v_{i+1}$ for $i=1, \dots, n-1$ and $e_n$ connects $v_n$ and $v_1$. 
We see that $\diTC G =2$ similarly to Example \ref{ex4}.

As we have seen already, $\diTC{\diS1}=2$. 
\item $b_{1}(G)\ge 2$. Then $\tc(G)=3$ (see above) and hence $\diTC G \ge 3$. On the other hand, $\diTC G \le 3$ by Proposition \ref{generalgraphs}. 
Thus $\diTC{G}=3$. 
\end{itemize}
\end{proof}

\section{Higher-dimensional directed spaces}


We begin by recalling the definition of ``geometric'' precubical
sets~\cite{Fajstrup:2005:DDC:2610563.2610596}. 
The 
interest \cite{bookdag} of such precubical sets is
that the precubical semantics of most programs is a geometric precubical set. 
Also they are sufficiently tractacle for us to compute, in some cases, their directed
topological complexity, or more precisely, the directed topological complexity of
their directed geometric realization, that we call, cubical complexes
(see Definition \ref{def:cubical-complex}). 

\begin{definition}
  \label{def:geometric-pcs}
  A precubical set~$C$ is geometric when it satisfies the following
  conditions:
  \begin{enumerate}
  \item \emph{no self-intersection}: two distinct iterated faces of a cube in
    $C$ are distinct
  \item \emph{maximal common faces}: two cubes admitting a common face admit a
    maximal common face.
  \end{enumerate}
\end{definition}

\begin{definition}
  \label{def:cubical-complex}
  A cubical complex is $K$ is a topological space of the form
  \[
  K=\left(\bigsqcup_{\lambda\in\Lambda}I^{n_\lambda}\right)/{\approx}
  \]
  where $\Lambda$ is a set, $(n_\lambda)_{\lambda\in\Lambda}$ is a family of
  integers, and $\approx$ is an equivalence relation, such that, writing
  $p_\lambda:I^{n_\lambda}\to K$ for the restriction of the quotient map
  $\bigsqcup_{\lambda\in\Lambda}I^{n_\lambda}\to K$, we have
  \begin{enumerate}
  \item for every $\lambda\in\Lambda$, the map $p_\lambda$ is injective,
  \item given $\lambda,\mu\in\Lambda$, if $p_\lambda(I^{n_\lambda})\cap
    p_\mu(I^{n_\mu})\neq\emptyset$ then there is an isometry from a face
    $J_\lambda$ of~$I^{n_\lambda}$ to a face $J_\mu$ of $I^{n_\mu}$ such that
    $p_\lambda(x)=p_\mu(y)$ if and only if $y=h_{\lambda,\mu}(x)$.
  \end{enumerate}
\end{definition}

As shown in \cite{cat0}~: 

\begin{proposition}
  \label{prop:geometric-pcs-real}
  The realization of a geometric precubical set is a cubical complex.
\end{proposition}

Generalising Proposition \ref{generalgraphs} one may show that $$\diTC{X}\leq 2\dim(X)+1$$ for nice cubical complexes $X$.
We shall address this question elsewhere. 

\subsection{The directed spheres}

Let $\Box^{n}$ be the cartesian
product of $n$ copies of the unit segment with the d-structure generated by the standard ordering on $[0,1]$. Its d-space structure is generated by
a partially-ordered space \cite{bookdag}. 

\begin{definition}
The directed sphere  $\diSn$ of dimension $n$  is defined as the boundary $\partial \Box^{n+1}$ of the hypercube $\Box^{n+1}$. Its d-structure is inherited
from the one of $\Box^{n+1}$. 
\end{definition}

\begin{proposition} 
$\diTC \diSn =2$ for any $n\ge 1$. 
\end{proposition}
The case $n=1$ is covered by Example \ref{ex2}; see \cite{Grant} for 
the general case.

\section{Directed homotopy equivalence and topological complexity}



As for now, there is no uniquely well-established notion of directed homotopy equivalence
between directed spaces, although there has been numeral proposals, among which one
linked to our present problem \cite{ArXiv}. 

We take the view here that directed homotopy equivalences should at least induce equivalent
trace categories, viewed with enough structure. We will show in the following sections that directed topological
complexity is an invariant of simple equivalences that should be implied by any ``reasonable'' directed equivalences. 

\subsection{A simple dihomotopy equivalence, and dicontractibility}

In \cite{ArXiv}, one of the authors introduced a notion of dihomotopy equivalence. The most important ingredient
are that two equivalent d-spaces should be homotopy equivalent in some naive way, and their trace spaces should be
homotopy equivalent as well\footnote{In \cite{ArXiv}, an extra ``bisimulation relation'' was added to the definition,
that we do not use here.}. First, we need to define continuous gradings~:
\begin{definition}
Let $v, v' \in V$, $q \ : \ U \rightarrow V$ be a continuous map, and $W \subseteq U\times U$ 
be the inverse image by $q \times q$ of $(v,v')$. 
Suppose we have a map $$h \ : \ PV(v,v')\times W \rightarrow PU$$ 
which is continuous and is 
such that for all $(u,u')\in W$, $h(p,u,u') \in PU(u,u')$. 

In this case, we say that $h$ is continuously graded over $W$, and
by abuse of notation, we write this graded map as a $h \ : \ PV \multimap PU$ given by grading
$h_{u,u'} \ : \ PV(q(u),q(u'))\rightarrow PU(u,u')$, varying continuously for $(u,u')\in W$ in $PU^{PV(v,v')}$, with the compact-open topology. 
\end{definition}


Any reasonable dihomotopy equivalence should be in particular a d-map inducing a (classical) homotopy equivalence
that also induced (classical) homotopy equivalences on the corresponding path spaces. We call this minimum requirement,
a simple dihomotopy equivalence~: 

\begin{definition}
\label{def:dihomotopyequiv}
Let $X$ and $Y$ be two d-spaces. A simple dihomotopy equivalence between $X$ and $Y$ is 
a d-map $f \ : \ X \rightarrow Y$ such that~: 
\begin{itemize}
\item $f$ is a d-homotopy equivalence between $X$ and $Y$, i.e. a homotopy equivalence with homotopy inverse 
a d-map $g \ : \ Y \rightarrow X$. 
\item 
There exists a map $F \ : \ PY \multimap PX$,  
continuously graded by
$F_{x,x'} \ : \ PY(f(x),f(x')) \rightarrow PX(x,x')$
for $(x,x') \in \Gamma_X$, 
such that $(Pf_{x,x'},F_{x,x'})$ is a homotopy equivalence\footnote{$Pf$ is the map
on paths which is the natural pointwise extension of $f$, i.e. $Pf(u)$ is the path
$t\rightarrow f(u(t))$ when $u$ is a path in $X$.} between $PX(x,x')$ and
$PY(f(x),f(x'))$ 
\item 
There exists a map $G \ : \ PX \multimap PY$, 
continuously graded by 
$G_{y,y'} \ : \ PX(g(y),g(y')) \rightarrow PY(y,y')$ 
for $(y,y') \in \Gamma_Y$
such that $(Pg_{y,y'},G_{y,y'})$ is a homotopy equivalence between
$PY(y,y')$ and $PX(g(y),g(y'))$. 
\end{itemize}
\end{definition}


We sometimes write $(f,g,F,G)$ for the full data associated to the simple dihomotopy
equivalence $f \ : \ X \rightarrow Y$. 

\paragraph{Remark : }
This definition clearly bears a lot of similarities with Dwyer-Kan weak equivalences
in simplicial categories (see e.g. \cite{bergner04}). The main ingredient
of Dwyer-Kan weak equivalences being exactly that $Pf$ induces a homotopy 
equivalence. But our definition adds continuity and directedness requirements 
which are instrumental to our theorems and to the classification
of the underlying directed geometry.

\begin{example}
{\rm \begin{itemize}
\item Let $X$, $Y$ be two directed spaces. Suppose $X$ and $Y$ are isomorphic as d-spaces i.e. that
there exists $f \ : \ X \rightarrow Y$ a dmap,
which has an inverse, also a dmap. Then $X$ and $Y$ are simply directed homotopy equivalent.
The proof goes as follows. 
Take $f=u$, $g=u^{-1}$, $Pg=F$ the pointwise application
of $u^{-1}$ on paths in $Y$ and $Pf=G$ the pointwise application of $u$ on
paths in $X$. This data obviously forms a directed homotopy equivalence.
\item The directed unit segment $\overrightarrow{I}$ is simply dihomotopically equivalent to a point. 
Consider the unique map $f \ : \ \overrightarrow{I} \rightarrow \{*\}$, and $g \ : \ \{*\} \rightarrow \overrightarrow{I}$ (the inclusion of the point
as $0$ in $\overrightarrow{I}$). Define $F \ : \ P\{*\} \rightarrow P\overrightarrow{I}$ by $F(*)$ being the constant map on $0$ and
$G \ : \ P\overrightarrow{I} \rightarrow P\{*\}$ to be the unique possible map (since $P\{*\}$ is a singleton $\{*\}$). 
\end{itemize}
}
\end{example}

As expected, directed topological complexity is an invariant of simple dihomotopy equivalence~: 

\begin{proposition}
\label{dihominvariance}
Let $X$ and $Y$ be two simply dihomotopically equivalent d-spaces. Then $\diTC{X}=\diTC{Y}$. 
\end{proposition}

\begin{proof}
As $X$ and $Y$ are dihomotopy equivalent, we have $f \ : \ X \rightarrow Y$ and 
$g \ : \ Y \rightarrow X$ dmaps, which form a homotopy equivalence between $X$ and $Y$. 
We also get $G$ a continuously graded map from $PX$ to $PY$, which can be restricted to 
$G_{y,y'} : PX(g(y),g(y')) \rightarrow PY(y,y')$, inverse modulo homotopy
to $Pg_{y,y'}$ ; and $F$ a continuously graded map from $PY$ to $PX$ such that its restrictions to $PX(x,x')$, for $(x,x')\in
\Gamma_X$, $F_{x,x'} \ : \ PX(x,x')
\rightarrow PY(f(x),f(x'))$ is inverse modulo homotopy to $Pf_{x,x'}$.

Suppose first $k=\diTC{X}$. Thus we can write $\Gamma_X=F^X_1 \cup \ldots \cup F^X_k$ such 
that we have a map $s \ : \ \Gamma_X \rightarrow PX$ with $\chi \circ s=Id$ and $s_{|F^X_i}$
is continuous. 

Define $F^Y_i=\{ u \in \Gamma_Y \ | \ g(u)\in F^X_i\}$ (which is either empty or an ENR 
as $F^X_i$
is ENR and $g$ is continuous) and define $t_{|F^Y_i}(u)=G_u \circ s_{|F^X_i}
\circ g(u)\in PY(u)$ for all $u \in F^Y_i \subseteq \Gamma_Y$. This is a continuous map in $u$ since
$s_{|F^X_i}$ is continuous, $g$ is continuous, and $G$ is continuous and graded. Therefore
$\diTC{Y} \leq \diTC{X}$. 

Conversely, suppose $l:\diTC{Y}$, $\Gamma_Y=F^Y_1 \cup \ldots \cup F^Y_l$
such 
that we have a map $t \ : \ \Gamma_Y \rightarrow PY$ with $\chi \circ t=Id$ and $t_{|F^Y_i}$
is continuous. Now define
$F^X_i=\{ u \in \Gamma_X \ | \ f(u)\in F^Y_i\}$ (which is either empty or an ENR
as $F^Y_i$ is ENR and $f$ is continous) 
and define $s_{|F^X_i}(u)=F_u \circ t_{|F^Y_i}
\circ f(u)\in PX(u)$ for all $u \in F^X_i \subseteq \Gamma_X$. This is a continuous map in $u$ since
$t_{|F^Y_i}$ is continuous, $f$ is continuous, and $F$ is continuous and graded. Therefore
$\diTC{X} \leq \diTC{Y}$. Hence we conclude that $\diTC{X}=\diTC{Y}$ and directed topological
complexity is an invariant of dihomotopy equivalence.  
\end{proof}

A very simple application is that some spaces must have directed topological complexity of 1~: 

\begin{definition}
A d-space $X$ is dicontractible if it is dihomotopically equivalent to a point. 
\end{definition}

By applying Proposition \ref{dihominvariance}, as the directed topological complexity of a point is 1, 
all dicontractible spaces have complexity 1, as in the undirected case. Similarly to the undirected case again, 
although with extra conditions, the converse is also true~:  

\begin{theorem}
\label{thm:contractible}
Suppose $X$ is a contractible d-space. 
Then, the 
dipath space map has a continuous section if and only if $X$ is dicontractible. 
\end{theorem}

\begin{proof}
As $X$ is contractible, we have $f: X \rightarrow \{a_0\}$ (the constant map) and $g: \{a_0\} \rightarrow X$ (the inclusion)
which form a (classical) homotopy equivalence. Trivially, $f$ and $g$ are dmaps, and form a
d-homotopy equivalence.  

Suppose that we have a continuous section $s$ of $\chi$. 
There is an obvious inclusion map $i~: \{s(a,b)\} \rightarrow \pth{X}(a,b)$, which is 
graded in $a$ and $b$. Define $R$ to be this map. 
Now the constant map $r~: \pth{X}(a,b) \rightarrow \{s(a,b)\}$ is a retraction map for $i$. 

We define 
$$ \begin{array}{lccc}
H \ : & \pth{X}\times [0,1] & \rightarrow & \pth{X} \\
& (u,t) & \rightarrow & \mbox{$v$ s.t. $\left\{\begin{array}{rcll}
v(x) & = & u(x) & \mbox{if $0\leq x \leq \frac{t}{2}$} \\
v(x) & = & s\left(u\left(\frac{t}{2}\right),u\left(1-\frac{t}{2}\right)\right)\left(\frac{x-\frac{t}{2}}{1-t}\right) & \mbox{if $\frac{t}{2} \leq x \leq 1-\frac{t}{2}$}\\
v(x) & = & u(x) & \mbox{if $1-\frac{t}{2}\leq x \leq 1$} \\
\end{array}\right.$}\\
\end{array}
$$
\noindent ($H(u,t)$ is extended by continuity for $t=1$ as being equal to $u$)

As concatenation and evaluation are continuous and as $s$ is continuous in both arguments
$H$ is continuous in 
$u \in PX$ and in $t$. $H$ induces families $H_{a,b} \ :  \pth{X}(a,b) \times [0,1] \rightarrow \pth{X}(a,b)$, and because $H$ is continuous in $u$ in the compact-open
topology, this family $H_{a,b}$ is continuous in $a$ and $b$ in $X$. 

Finally, we note that 
$H(u,1)=u$ and $H(u,0)=s(u(0),u(1))=i\circ r(u)$. Hence $r$ is a deformation retraction and
$PX(a,b)$ is homotopy equivalent to $\{s(a,b)\}$ and has the homotopy type we expect
(is contractible for all $a$ and $b$), meaning that $R$ is a (graded) homotopy equivalence.

Conversely, suppose $X$ is dicontractible. 
We have in particular a continuous map 
$R \ : \ \{ * \} \rightarrow \pth{X}$, which is graded in $(a,b)\in \Gamma_X$. Define 
$s(a,b)=R_{a,b}(*)$, this is a continuous section of $\chi$. 
\end{proof}

\paragraph{Remark : }
Sometimes, we do not know right away, in the theorem above, that $X$ is contractible. But instead, 
there is  an initial state in $X$, i.e. a state $a_0$
from which 
every point of $X$ is reachable. 
Suppose then that, as in the Theorem above, $\chi$ has a continuous section $s : \Gamma_X \rightarrow PX$. 
Consider $s'(a,b)=s^{-1}(a_0,a)*s(a_0,b)$ the concatenation of the inverse dipath,
going from $a$ to $a_0$, with the dipath going from $a_0$ to $b$~: this is a continuous
path from $a$ to $b$ for all $a$, $b$ in $X$. Now, $s'$ is obviously continuous since
concatenation, and $s$, are. By a classical theorem \cite{farber}, this implies that
$X$ is contractible and the rest of the theorem holds. 

\begin{example}
{\rm Direct applications of Proposition \ref{dihominvariance} show that~: 
\begin{itemize}
\item Directed $n$-tori $\diOo^n$ and $\diOo^m$ 
cannot be simply dihomotopically equivalent when $n \neq m$. 
\item Directed $n$-tori $\diOo^n$ cannot be dihomotopically equivalent to 
any directed graph for $n\geq 3$. 
\end{itemize}
}
\end{example}

\subsection{Natural homology, and dicontractibility}


We now come to make a first connection between some invariants that have been 
introducted in directed topology (see e.g. \cite{naturalhomology}), like natural homology, 
\cite{Eilenberg}. 

We first recap the construction of such invariants. 

A \emph{monotonic reparametrization $r$} is a monotonic continuous surjection from $[0,1]$ to $[0,1]$.

Let $X$ be a pospace, i.e. a topological space together with a closed order
$\leq \subseteq X \times X$. $X$ is then a particular d-space with the directed
paths being the continuous and increasing maps from the unit segment, with the standard
ordering, to $X$. 

Let now $p$ and $q$ two dipaths from $a$ to $b$ in $X$. 
We say that \emph{$p$ is reparametrized in $q$} if there exists a monotonic reparametrization $\gamma$ such that $p\circ\gamma = q$. The \emph{trace} of $p$, written $\langle p \rangle$ is the equivalence class modulo monotonic reparametrization.

Now we can put together all dipaths from point $a$ to point $b$, modulo monotonic reparame\-trization in a topological space:

Let $X$ be a pospace and $a$ and $b \in X$. We topologize 
the set of traces of dipaths from $a$ to $b$, 
with the compact-open topology. Its quotient $\trace{X}{a}{b}$ 
by reparametrization, with the quotient topology is called the \emph{trace space in $X$ from $a$ to $b$} (see \cite{Raussen}).

\begin{definition}
\label{topologicaltracecat}
We define $\T_X$ to be the category whose:
\begin{itemize}
	\item objects are traces of $X$
	\item morphisms (also called extensions) from $\langle p \rangle$ to $\langle q \rangle$ with $p$, a dipath from $x$ to $y$ and $q$, one from $x'$ to $y'$ are pairs of traces $(\langle \alpha \rangle,\langle \beta \rangle)$ such that $\langle q \rangle = \langle \alpha \star p \star \beta \rangle$
\end{itemize}
We then define $\map{\overrightarrow{T}_{*}(X)}{\T_X}{\textbf{Top}_*}$ which maps:
\begin{itemize}
	\item every trace $\langle p \rangle$ with $p$ from $x$ to $y$ to the pointed space $(\trace{X}{x}{y},\langle p \rangle)$
	\item every extension $(\langle \alpha \rangle, \langle \beta \rangle)$ with $\alpha$ dipath from $x'$ to $x$ and $\beta$ dipath from $y$ to $y'$ to the continuous map $\map{\langle \alpha \star \_ \star \beta \rangle}{\trace{X}{x}{y}}{\trace{X}{x'}{y'}}$ which maps $\langle p \rangle$ to $\langle \alpha \star p \star \beta \rangle$.
\end{itemize} 
\end{definition}

We can now define the natural homology functors~: 

\begin{definition}[Natural homology]
\label{naturalhomology}
We define for $n\geq 1$, $\map{\sysh{n}{X}}{\T_X}{\M}$ (where $\M$ is $\ab$) composing $\overrightarrow{T}_{*}(X)$ with the $(n-1)^{th}$ homology group functor $H_{n-1}$. 
\end{definition}

\begin{remark}
$\T_X$ is actually the category of factorization (or twisted arrow category) 
of the category whose objects are points of $X$ and morphisms are traces and this makes $\overrightarrow{T}_{*}(X)$ into a natural system in the sense of \cite{bauwir}.
\end{remark}

\begin{example} 
\label{exa:syshom} {\rm (taken from \cite{Eilenberg})
We consider the pospace $\diS1$ again, which is made up
of two directed segments $a$ and $b$ where there initial points are
identified, and their final points are identified too. In the following
picture, we distinguish two particular points $x$
and $y$ on $a$, with $x < y$ (respectively $x'$ and $y'$ on $b$, with
$x' < y'$), which we will use
to describe the category of factorization $\T_{a+b}$ as well as the
natural homology $\sysh{n}{a+b}$. 

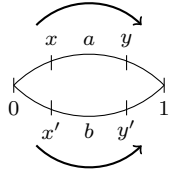
\begin{center}
	\begin{tikzpicture}
		\node (1) at (0,-0.3) {\scriptsize{$0$}};
		\node (2) at (2,-0.3) {\scriptsize{$1$}};
		\node (a) at (1,0.6) {\scriptsize{$a$}};
		\node (b) at (1,-0.6) {\scriptsize{$b$}};
		\node (x) at (0.5,0.6) {\scriptsize{$x$}};
		\node (y) at (1.5,0.6) {\scriptsize{$y$}};
		\node (x') at (0.5,-0.6) {\scriptsize{$x'$}};
		\node (y') at (1.5,-0.6) {\scriptsize{$y'$}};
		\draw (0,0) to [bend left = 45] (2,0);
		\draw (0,0) to [bend right = 45] (2,0);
		\draw (0.5,0.2) -- (0.5,0.4);
		\draw (1.5,0.2) -- (1.5,0.4);
		\draw (0.5,-0.2) -- (0.5,-0.4);
		\draw (1.5,-0.2) -- (1.5,-0.4);
		\draw (0,-0.1) -- (0,0.1);
		\draw (2,-0.1) -- (2,0.1);
		\draw[->,thick] (0.3,0.8) to [bend left = 45] (1.7,0.8);
		\draw[->,thick] (0.3,-0.8) to [bend right = 45] (1.7,-0.8);
	\end{tikzpicture}
	\end{center}

The description of $\T_{a+b}$ is now as follows. Objects of $\T_{a+b}$ are
dipaths, which can be either:
\begin{itemize}
\item constant dipaths, $0$, $x$, $y$, $x'$, $y'$, $1$, 
for all points $x$, $y$, $x'$, $y'$ that we chose to distinguish 
in the picture
of $a+b$.
\item non constant and non maximal 
dipaths of the form $[0,x]$, $[x,y]$, $[y,1]$ etc.
\item maximal dipaths $a$ and $b$
\end{itemize}

We chose below to draw a picture of a subcategory of $\T_{a+b}$, where
$x$, $y$, $x'$ and $y'$ are any distinguished points of $a$ and $b$ 
as discussed before. The extension morphisms in $\T_{a+b}$ are pictured
below as arrows ; for instance, there is an extension morphism from 
dipath $[x,y]$ to $[0,y]$ and to $[x,1]$, among other extension morphisms: 

	
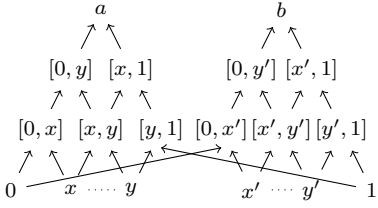
\begin{center}
	\begin{tikzpicture}[scale = 0.8]
		\node (0) at (0,0) {\scriptsize{$0$}};
		\node (x) at (1,0) {\scriptsize{$x$}};
		\node (y) at (2,0) {\scriptsize{$y$}};
		\node (0x) at (0.5,1) {\scriptsize{$[0,x]$}};
		\node (y1) at (2.5,1) {\scriptsize{$[y,1]$}};
		\node (xy) at (1.5,1) {\scriptsize{$[x,y]$}};
		\node (0y) at (1,2) {\scriptsize{$[0,y]$}};
		\node (x1) at (2,2) {\scriptsize{$[x,1]$}};
		\node (a) at (1.5,3) {\scriptsize{$a$}};
		
		\node (1) at (6,0) {\scriptsize{$1$}};
		\node (x') at (4,0) {\scriptsize{$x'$}};
		\node (y') at (5,0) {\scriptsize{$y'$}};
		\node (0x') at (3.5,1) {\scriptsize{$[0,x']$}};
		\node (y'1) at (5.5,1) {\scriptsize{$[y',1]$}};
		\node (x'y') at (4.5,1) {\scriptsize{$[x',y']$}};
		\node (0y') at (4,2) {\scriptsize{$[0,y']$}};
		\node (x'1) at (5,2) {\scriptsize{$[x',1]$}};
		\node (b) at (4.5,3) {\scriptsize{$b$}};
		
		\draw[dotted] (y) -- (x);
		\draw[dotted] (y') -- (x');
		
		\draw[->] (0) -> (0x);
		\draw[->] (x) -> (0x);
		\draw[->] (1) -> (2.5,0.7);
		\draw[->] (y) -> (y1);
		\draw[->] (x) -> (xy);
		\draw[->] (y) -> (xy);
		\draw[->] (0x) -> (0y);
		\draw[->] (xy) -> (0y);
		\draw[->] (y1) -> (x1);
		\draw[->] (xy) -> (x1);
		\draw[->] (0y) -> (a);
		\draw[->] (x1) -> (a);
		
		\draw[->] (0) -> (3.5,0.7);
		\draw[->] (x') -> (0x');
		\draw[->] (1) -> (y'1);
		\draw[->] (y') -> (y'1);
		\draw[->] (x') -> (x'y');
		\draw[->] (y') -> (x'y');
		\draw[->] (0x') -> (0y');
		\draw[->] (x'y') -> (0y');
		\draw[->] (y'1) -> (x'1);
		\draw[->] (x'y') -> (x'1);
		\draw[->] (0y') -> (b);
		\draw[->] (x'1) -> (b);
	\end{tikzpicture}
	\end{center}

Now, we can picture a subdiagram of $\sysh{1}{a+b}$, by applying the
homology functor on the trace spaces from the starting point to the end
point of the dipaths, objects of $\T_{a+b}$. For instance, the trace
space $\trace{a+b}{x}{y}$ (respectively 
$\trace{a+b}{0}{y}$) corresponding to dipath $[x,y]$ 
(respectively $[0,y]$) in the diagram
above, is just a point, hence has zeroth homology group equal to $\Z$
(respectively $\Z$). 
All other zeroth homology groups are trivial with the exception of 
the ones corresponding to the two maximal dipaths (up to reparametrization) 
$a$ and $b$, going from $0$ to $1$. In that case, 
$\trace{a+b}{0}{1}$ is composed of two points, that we can identify with
$a$ and $b$, and has $\Z^2$ (or $\Z[a,b]$ with the identification we just
made) as
zeroth homology. Now the extension morphism from $[0,y]$ to $a$ induces
a map in homology which maps the only generator of $H_0(\trace{a+b}{0}{y})$
to generator $a$ in $\Z[a,b]$ as indicated in the picture below: 

	
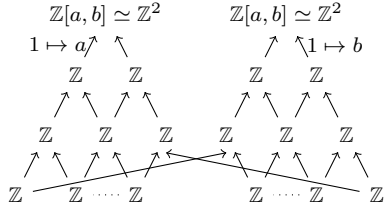
\begin{center}
	\begin{tikzpicture}[scale = 0.8]
		\node (0) at (0,0) {\scriptsize{$\mathbb{Z}$}};
		\node (x) at (1,0) {\scriptsize{$\mathbb{Z}$}};
		\node (y) at (2,0) {\scriptsize{$\mathbb{Z}$}};
		\node (0x) at (0.5,1) {\scriptsize{$\mathbb{Z}$}};
		\node (y1) at (2.5,1) {\scriptsize{$\mathbb{Z}$}};
		\node (xy) at (1.5,1) {\scriptsize{$\mathbb{Z}$}};
		\node (0y) at (1,2) {\scriptsize{$\mathbb{Z}$}};
		\node (x1) at (2,2) {\scriptsize{$\mathbb{Z}$}};
		\node (a) at (1.5,3) {\scriptsize{$\mathbb{Z}[a,b]\simeq\mathbb{Z}^2$}};
		
		\node (1) at (6,0) {\scriptsize{$\mathbb{Z}$}};
		\node (x') at (4,0) {\scriptsize{$\mathbb{Z}$}};
		\node (y') at (5,0) {\scriptsize{$\mathbb{Z}$}};
		\node (0x') at (3.5,1) {\scriptsize{$\mathbb{Z}$}};
		\node (y'1) at (5.5,1) {\scriptsize{$\mathbb{Z}$}};
		\node (x'y') at (4.5,1) {\scriptsize{$\mathbb{Z}$}};
		\node (0y') at (4,2) {\scriptsize{$\mathbb{Z}$}};
		\node (x'1) at (5,2) {\scriptsize{$\mathbb{Z}$}};
		\node (b) at (4.5,3) {\scriptsize{$\mathbb{Z}[a,b]\simeq\mathbb{Z}^2$}};
		
		\node at (0.7,2.5) {\scriptsize{$1 \mapsto a$}};
		\node at (5.3,2.5) {\scriptsize{$1 \mapsto b$}};
		
		\draw[dotted] (y) -- (x);
		\draw[dotted] (y') -- (x');
		
		\draw[->] (0) -> (0x);
		\draw[->] (x) -> (0x);
		\draw[->] (1) -> (2.5,0.7);
		\draw[->] (y) -> (y1);
		\draw[->] (x) -> (xy);
		\draw[->] (y) -> (xy);
		\draw[->] (0x) -> (0y);
		\draw[->] (xy) -> (0y);
		\draw[->] (y1) -> (x1);
		\draw[->] (xy) -> (x1);
		\draw[->] (0y) -> (a);
		\draw[->] (x1) -> (a);
		
		\draw[->] (0) -> (3.5,0.7);
		\draw[->] (x') -> (0x');
		\draw[->] (1) -> (y'1);
		\draw[->] (y') -> (y'1);
		\draw[->] (x') -> (x'y');
		\draw[->] (y') -> (x'y');
		\draw[->] (0x') -> (0y');
		\draw[->] (x'y') -> (0y');
		\draw[->] (y'1) -> (x'1);
		\draw[->] (x'y') -> (x'1);
		\draw[->] (0y') -> (b);
		\draw[->] (x'1) -> (b);
	\end{tikzpicture}
	\end{center}
}
\end{example}


We now define bisimulation as in \cite{naturalhomology}. A bisimulation between functor categories into Abelian groups 
$P \ : \ F \rightarrow Ab$, 
and 
$Q \ : \ G \rightarrow Ab$ is
a ``relation'' labelled with such isomorphisms of Abelian groups, i.e.
is a set of triples $$(\sigma,\eta,\tau)$$ \noindent which is hereditary in 
the following sense~:  

\begin{itemize}
\item
for all $\langle \alpha,\beta\rangle \in F$ from $x$ to $x'$, 
if $(x, \eta, y) \in R$, 
there exists $\langle \gamma,\delta \rangle \in 
  G$ from $y$ to $y'$  
such that 
$(x',\eta', y')\in R$
and such that the following diagram commutes~: 

\begin{center}
  \begin{tikzpicture}[scale=1]
\matrix (m) [matrix of math nodes,row sep=3em,column sep=4em,minimum width=2em]
  {
     \scriptstyle P(x) & \scriptstyle Q(y) \\
     \scriptstyle P(x') & \scriptstyle Q(y') \\};
  \path[-stealth]
    (m-1-1) edge node [left] {$\langle \alpha,\beta \rangle$} (m-2-1)
            edge node [above] {${\eta}$} (m-1-2)
(m-2-1) edge node [above] {${\eta}$} (m-2-2)
    (m-1-2) edge node [dashed,right] {$\langle \gamma,\delta \rangle$} (m-2-2);
\end{tikzpicture}
\end{center}
\item 
for all 
$\langle \gamma,\delta\rangle \in G$ from $y$ to $y'$, 
if $(x, \eta, y) \in R$, 
there exists $\langle \alpha,\beta\rangle \in F$ from $x$ to $x'$
such that 
$(x',\eta', y') \in R$ 
and 
such that the following diagram, as above, commutes 
up 
to homotopy
\begin{center}
  \begin{tikzpicture}[scale=1]
\matrix (m) [matrix of math nodes,row sep=3em,column sep=4em,minimum width=2em]
  {
     \scriptstyle P(x) & \scriptstyle Q(y) \\
     \scriptstyle P(x') & \scriptstyle Q(y') \\};
  \path[-stealth]
    (m-1-1) edge node [left] {$\langle \alpha,\beta\rangle$} (m-2-1)
            edge node [above] {${\eta}$} (m-1-2)
(m-2-1) edge node [above] {${\eta'}$} (m-2-2)
    (m-1-2) edge node [dashed,right] {$\langle \gamma,\delta\rangle$} (m-2-2);
\end{tikzpicture}
\end{center} 
\end{itemize}

The main connection with directed topological complexity is as follows~: 

\begin{proposition}
Let $X$ be a d-space. $X$ has directed topological complexity of one (i.e. is
dicontractible) implies that its natural homologies $\sysh{n}{X}$ are all 
bisimulation equivalent to either, $1_\Z \ : \ {\bf 1} \rightarrow \Z$ for
$n=1$, or to $1_0 \ : \ {\bf 1} \rightarrow 0$ for $n > 1$, defined as~:
\begin{itemize}
\item {\bf 1} is the terminal category, with one object 1 and one morphism (the identity
on 1)
\item $1_\Z(1)=\Z$, $1_0(1)=0$.
\end{itemize}
\end{proposition}

\begin{proof}
Suppose that $X$ has directed topological complexity of 1. Then by Theorem 
\ref{thm:contractible}, all trace spaces $\trace{X}{x}{y}$ are contractible, for all
$(x,y) \in \Gamma_X$, hence
$\sysh{1}{X}(x,y)=\Z$ and $\sysh{n}{X}(x,y)=0$ for $n > 1$. Therefore the natural
homology functors are all constant, either with value $\Z$ or with value 0, and it
is a simple exercise to see that the relation between $\T_X$ and ${\bf 1}$ which 
relates all objects of $\T_X$ to the only object 1 of ${\bf 1}$ is hereditary, hence
is a bisimulation equivalence.
\end{proof}

\begin{example}
{\rm 
We get back to example $\diS1$. Its first homology functor was calculated in Example
\ref{exa:syshom} and is not a constant functor (it contains $\Z^2$ and $\Z$ in its
image). Therefore $\diS1$ cannot have directed topological complexity of 1. It is also
easy to see that the first natural homology functor of $\diOo$ is $\Z^\N$ between two
equal points and hence
cannot have directed topological complexity of 1. 
}
\end{example}

%

\end{document}